\documentclass[12pt,letterpaper]{article}

\usepackage[utf8]{inputenc}

\usepackage{amsmath}
\usepackage{amsfonts}
\usepackage{amsthm}
\usepackage{amssymb}
\usepackage{array}
\usepackage{caption}
\usepackage{color}
\usepackage{enumitem}
\usepackage{epic,eepic}
\usepackage{graphics}
\usepackage{graphicx}
\usepackage{hyperref}
\usepackage{mathtools}
\usepackage{rotating}
\usepackage{verbatim}
\usepackage{mathtools}

\newtheorem{prop}{Proposition}

\newtheorem{theorem}[prop]{Theorem}

\begin{document}

\title{The Sierpinski Triangle and The Ulam-Warburton Automaton}
\author{Tanya Khovanova, Eric Nie, and Alok Puranik}
\maketitle

\begin{abstract}
This paper is about the beauty of fractals and the surprising connections between them. We will explain the pioneering role that the Sierpinski triangle plays in the Ulam-Warburton automata and show you a number of pictures along the way.
\end{abstract}

\section{Introduction}

The Sierpinski triangle appeared in some Italian mosaics in the 13th century \cite{wikiST}. It was not named the Sierpinksi triangle then; in fact, Wac\l aw Sierpi\'{n}ski was not yet born. He was born in 1882 and described the triangle mathematically in 1915 \cite{Si}.

The Ulam-Warburton automaton was discovered much later. The first available description is in a paper by Ulam \cite{U}. We were not able to determine who Warburton is or find his works about the automaton. Later, in 1994, Richard Stanley rediscovered the automaton \cite{SC}. The automaton continued to fascinate scientists, and many other people wrote about it \cite{APS, PSX, S, W}.

The Ulam-Warburton automaton is a fractal, as is the Sierpinki gasket. But the connection between these two fractals is much deeper.

We stumbled upon this connection when we started studying the hexagonal analog of the Ulam-Warburton automaton. The project was a part of the MIT-PRIMES program and was suggested by Richard Stanley.

This paper is about the beauty of fractals and the surprising connections between them. We will explain the pioneering role that the gasket plays in the Ulam-Warburton automata and also show you many pictures.

\section{The Ulam-Warburton Automaton}

The Ulam-Warburton automaton grows on a square grid. It starts with a single cell, that is one unit square of the grid. A new cell is born if it is adjacent to exactly one live cell. In this particular automaton two cells are considered adjacent if they share a side. Live cells never die. Figure~\ref{fig:SquareFirst} shows the initial configuration and the live cells born during the next five generations.

\begin{figure}[htbp]
\centering
$
\begin{array}{ccc}
\includegraphics[height=25mm]{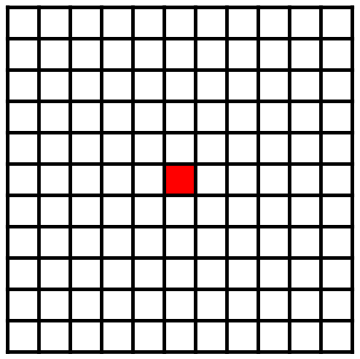} &
\includegraphics[height=25mm]{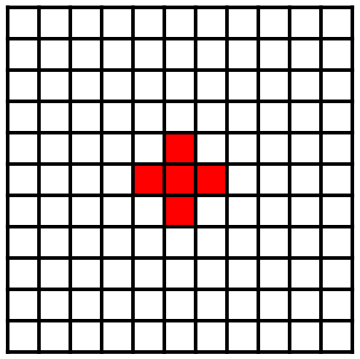} &
\includegraphics[height=25mm]{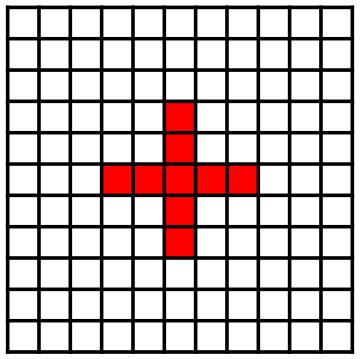}\\ 
\includegraphics[height=25mm]{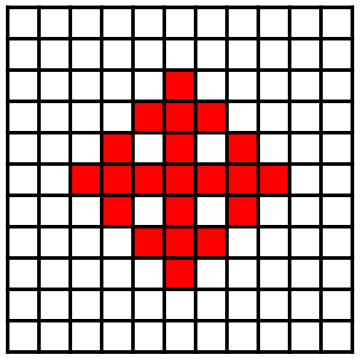} & 
\includegraphics[height=25mm]{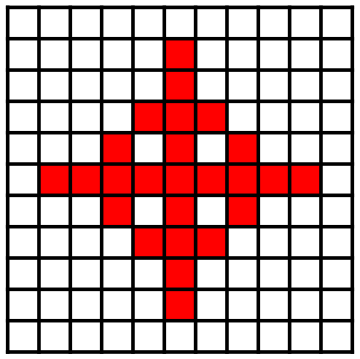} &
\includegraphics[height=25mm]{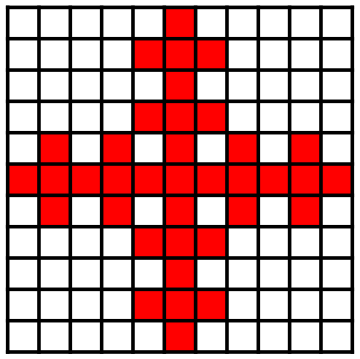}
\end{array}$
\caption{First generations of the Ulam-Warburton automaton}\label{fig:SquareFirst}
\end{figure}

The population of live cells becomes more elaborate with every generation. Figure~\ref{fig:square13-15} shows the set of live cells at generations 13 and 15.

\begin{figure}[htbp]
\centering$
\begin{array}{cc}
\includegraphics[height=50mm]{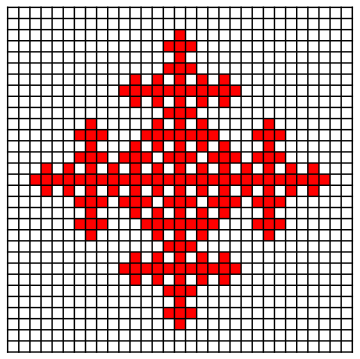} &
\includegraphics[height=50mm]{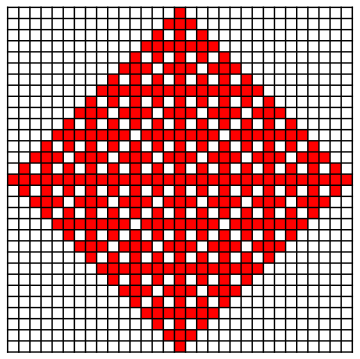}
\end{array}$
\caption{Generations 13 and 15 of the Ulam-Warburton automaton}\label{fig:square13-15}
\end{figure}

But where is the Sierpinski triangle? These pictures look squary while the triangle is triangular. The Sierpinski triangle is sometimes also called the Sierpinski sieve or the Sierpinski gasket. Sieve or gasket, the basic shape is still a triangle, as shown in Figure~\ref{fig:gasket}.

\begin{figure}[htbp]
\centering
\includegraphics[height=40mm]{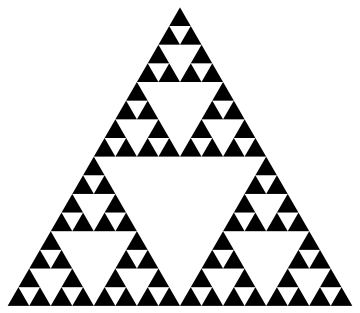}
\caption{Sierpinski gasket}\label{fig:gasket}
\end{figure}

The gasket is composed of triangles and looks triangular. However, we can imagine that the gasket has generations the same way the automaton does. We can consider the gasket row by row, with each row corresponding to a generation. The top row, row 0, has a single small triangle. 
Figure~\ref{fig:gasket13} shows the Sierpinski triangle after 13 generations. It looks unfinished and asymmetrical. The gasket looks like a complete triangle only for generations $2^k-1$. Now we see the first similarity: the Ulam-Warburton automaton looks like a complete square balancing on its corner only for the same generations $2^k-1$.

\begin{figure}[htbp]
\centering
\includegraphics[height=35mm]{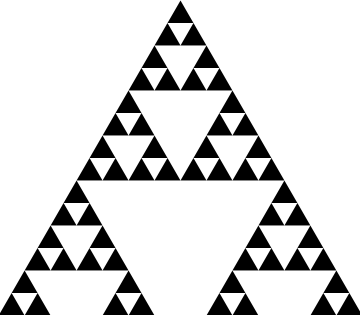}
\caption{Sierpinski gasket after 13 generations}\label{fig:gasket13}
\end{figure}

Before searching for the gasket in the Ulam-Warburton automaton let us try to generalize the automaton to the hexagonal grid.

\section{The Hex-Ulam-Warburton Automaton}

Our new automaton lives on the hexagonal grid, but it still follows the same rules. Again, the automaton starts with a single cell. A new cell is born if it is adjacent to exactly one live cell, and a live cell never dies. We call this automaton the Hex-Ulam-Warburton automaton, or Hex-UW automaton for short. Figure~\ref{fig:hexFirst} shows the initial configuration and the first five generations of the Hex-UW automaton. For disambiguation, we will sometimes call the original Ulam-Warburton automaton the Square-Ulam-Warburton automaton, or Square-UW automaton for short.

\begin{figure}[htbp]
\centering
$\begin{array}{ccc}
\includegraphics[height=25mm]{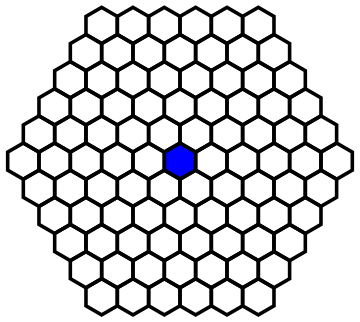} &
\includegraphics[height=25mm]{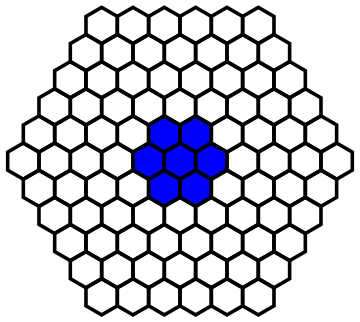} &
\includegraphics[height=25mm]{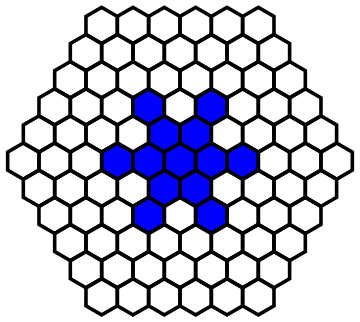}\\ 
\includegraphics[height=25mm]{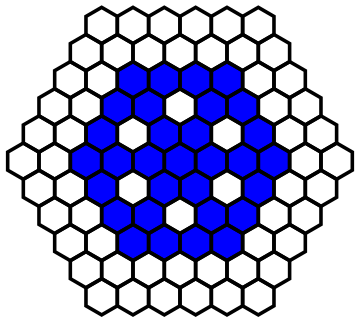} & 
\includegraphics[height=25mm]{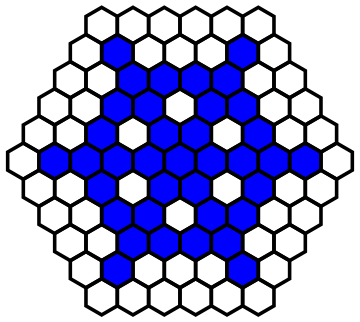} &
\includegraphics[height=25mm]{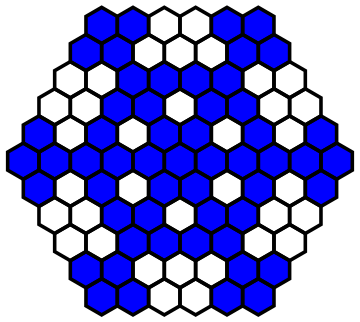}
\end{array}$
\caption{First generations of Ulam-Warburton-Hex Automaton}\label{fig:hexFirst}
\end{figure}

Similar to the Square-UW automaton the later generations of the Hex-UW automaton produce a more intricate picture. Figure~\ref{fig:hex13-15} shows generations 13 and 15.

\begin{figure}[htbp]
\centering$
\begin{array}{cc}
\includegraphics[height=50mm]{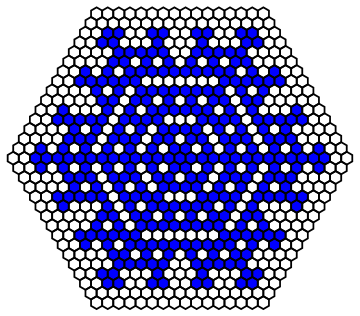} &
\includegraphics[height=50mm]{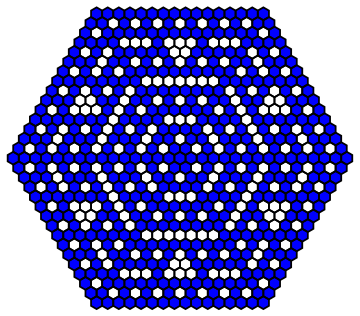}
\end{array}$
\caption{Generations 13 and 15 of the Hex-UW automaton}\label{fig:hex13-15}
\end{figure}

Like the Square-UW automaton, only generations $2^k-1$ look like a complete hexagon. The hexagon consists of 6 triangles, and each has a fractal structure reminiscent of the Sierpinski gasket. At generation 31, the gasket is even more visible, as in Figure~\ref{fig:hex31}, where we removed the background grid. But we still want to keep you in suspense about the role of the gasket and discuss the symmetries of our automata first.

\begin{figure}[htbp]
\centering
\includegraphics[height=50mm]{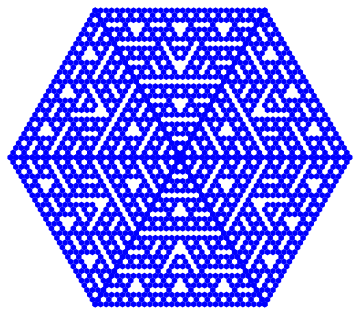}
\caption{Generation 31 of the Hex-UW automaton}\label{fig:hex31}
\end{figure}

\section{Symmetries}

Both the Square-UW and Hex-UW automata have symmetries because the initial position is symmetric and growth rules respect the symmetry. A square has 4 reflection lines and a hexagon has 6 reflection lines. The same lines passing through the initial cell are the lines of symmetry for all generations.

The symmetry considerations can be used to prove the following proposition that is true for both square and hex automata.

\begin{prop}
The cells whose centers lie on the line of reflectional symmetry which passes through the corners of the initial cell are never born.
\end{prop}

\begin{proof}
The neighbors of each cell on this line can be partitioned into symmetric pairs. This implies that the cells on these lines always have an even number of living neighbors at every point in time.
\end{proof}

\section{The Sierpinski gasket}

We promised to find the gasket in these automata. Let us divide the square grid into four quadrants originating at the initial cell, so that the overlap of a complete square at generation $2^k-1$ resembles a triangle. Similarly, we can divide the hexagonal grid into six parts.  We call each triangle at population $n$ a \textit{slice}. Thus the square automaton is divided into 4 slices and the hex automaton into 6 slices.

We can draw the gasket in each slice of the automaton. In Figure~\ref{fig:gasketInUWHex}, we marked the Sierpinski gasket with dots using the hex automaton as the background in the left picture. This way you can see that the dots are always in blue cells. In the right picture, the cells of the Sierpinski gasket are represented as solid black hexagons to emphasize the gasket structure. The figure illustrates 15 generations of the Hex-UW automaton and the Sierpinski triangle. The pictures help us formulate our first theorem which we will prove later.

\begin{figure}[htbp]
\centering$
\begin{array}{cc}
\includegraphics[height=50mm]{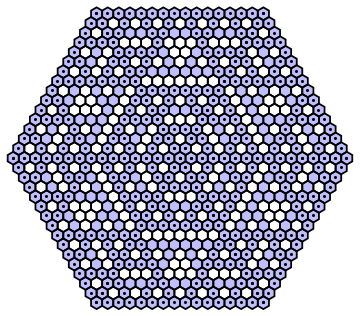} &
\includegraphics[height=50mm]{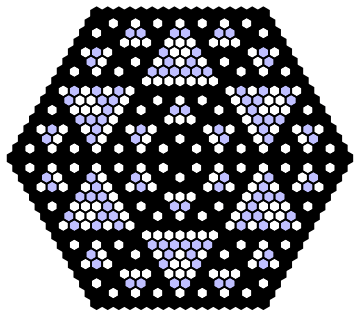}
\end{array}$
\caption{The Sierpinski gasket in the Hex-UW automaton}\label{fig:gasketInUWHex}
\end{figure}

\begin{theorem}
The Hex-UW automaton up to generation $n$ contains the Sierpinski gasket up to generation $n$.
\end{theorem}

Now we can go back to the square grid and repeat the process in each of the four slices of the square as in Figure~\ref{fig:gasketInUW} which represents 15 generations. In the figure to the left, the cells of the Sierpinski triangles are represented as dots, so the background of the Square-UW automaton can be seen clearly. In the figure to the right, the cells of the gasket are black for emphasis.

\begin{figure}[htbp]
\centering$
\begin{array}{cc}
\includegraphics[height=50mm]{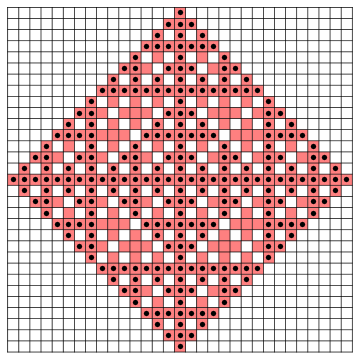} &
\includegraphics[height=50mm]{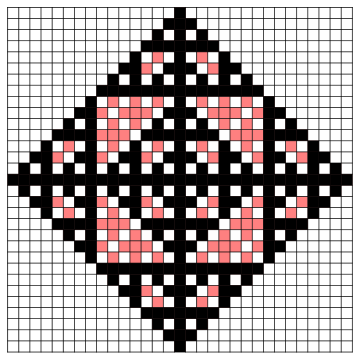}
\end{array}$
\caption{The Sierpinski gasket in the Square-UW automaton}\label{fig:gasketInUW}
\end{figure}

The is the square analog of the hex theorem above.

\begin{theorem}
The Square-UW automaton up to generation $n$ contains the Sierpinski gasket up to generation $n$.
\end{theorem}

 We will prove the theorems later; in addition to proving them, we would like to gain an understanding of the role that the gasket plays in the automata, but first we will discuss family trees. 
 
\section{Family}

We call the initial cell \textit{the patriarch}. The patriarch is generation 0 of both automata. The neighbors of the patriarch that are born in the first step are generation 1, and new cells that are born at step $n$ are generation $n$. There is a more complicated structure to the cells than just generations. Each cell, other than the patriarch, has exactly one parent. Indeed, each living non-patriarch cell was born because it was adjacent exactly to one other cell. Define that other cell, which was the progenitor of this cell, to be its \textit{parent}. Similarly, if cell $a$ is a parent of cell $b$, then $b$ is \textit{a child} of $a$.

\begin{prop}
If a cell is born in generation $n$, its parent is born in the previous generation, or generation $n-1$.
\end{prop}

Thus, we have a family. And families have family trees. Where are the trees here? It is traditional to draw a family tree as a graph, where parents are vertices and parent-children pairs are connected by an edge. For example, Figure~\ref{fig:square6tree} shows the Square-Ulam-Warburton automaton after 6 generations and the corresponding family tree.

\begin{figure}[htbp]
\centering$
\begin{array}{cc}
\includegraphics[height=45mm]{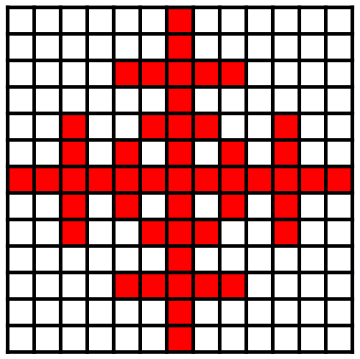} &
\includegraphics[height=45mm]{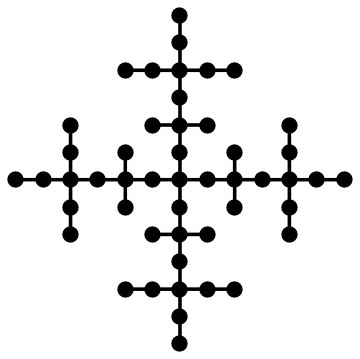}
\end{array}$
\caption{Generations 6 of the Ulam-Warburton automaton and its tree}\label{fig:square6tree}
\end{figure}

Let the sequence of parents from a cell to the patriarch be called its \textit{lineage}. Define an \textit{ancestor} of a cell to be a cell that is part of its lineage. For example, great-great-grandparents are ancestors and are part of a lineage.

By birth rules each cell has exactly one parent. But how many children can a cell have? 

If you look at the Square-Ulam-Warburton automaton, you can see that the patriarch has 4 children, and all other cells have either 0, 1 or 3 children. Clearly, a non-patriarch cannot have more than 3 children because it only has four neighbors with one of them being the parent. This automaton exhibits fractal behavior, which means self-similarity. For fractals in particular, if you only see 0, 1, or 3 children at several initial stages, you can be sure that future cells will not have 2 children. Figure~\ref{fig:UWchildren} color codes the fertility of cells. The parents without children are leaves in a family tree, and they are green in the Figure. Parents with three children are blue, and those with one child are red. The patriarch is black. On the border we colored cells with their potential fertility. Their children will appear in the next generation.

\begin{figure}[htbp]
\centering
\includegraphics[height=50mm]{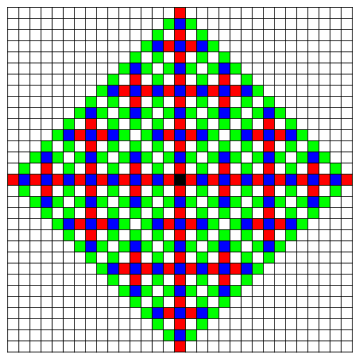}
\caption{Square-UW cells color-coded by the number of children}\label{fig:UWchildren}
\end{figure}

Similarly, the patriarch of the Hex-UW automaton has 6 children, and no other cell can have more than 5 children. In reality, the number of children is either 0, 1, 2, or 3. Figure~\ref{fig:UWHexchildren} color codes the fertility of cells. The parents without children are green. The parents with three children are blue, ones with two children are purple, and those with one child are red. The patriarch is black. On the border we colored cells with their expected fertility.

\begin{figure}[htbp]
\centering
\includegraphics[height=50mm]{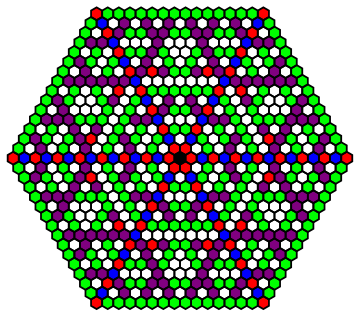}
\caption{Hex-UW cells color-coded by the number of children}\label{fig:UWHexchildren}
\end{figure}

\section{Distance}

We now forget automata for a second, and think about the distance between two cells on a grid. We can define a \textit{geometric distance} between two cells as a distance between their centers. 

But there is another special distance that plays a special role on grids called the \textit{Manhattan distance}. If you are in Manhattan trying to stroll from one place to another then you cannot cut through buildings; you have to use the streets. And streets in Manhattan are similar to a square grid. In the case of the automaton we have an additional constraint. The only move that is allowed is from the center of a cell, to the center of a neighboring cell. This move is considered to cover distance 1. This Manhattan distance between cells $a$ and $b$ is defined as the length of the shortest path that starts at cell $a$, ends at cell $b$, and is restricted as above. On the square grid, a neighbor of a neighbor might be at Manhattan distance 0 or 2. On the hexagonal grid, a neighbor of a neighbor might be at Manhattan distance 0, 1 or 2. The Manhattan distance is related to the geometric distance.

\begin{prop}
Consider the shortest path in a Manhattan sense that connects cells $a$ and $b$. If we follow this path from $a$ to $b$ then at each point in time the geometric distance from $a$ will increase.
\end{prop}

In other words, when we follow the shortest path in a Manhattan sense we never turn back toward the original, we always move away. 

Now let us turn back to our automata. The children are at Manhattan distance 1 from their parents.

\begin{prop}
A cell born in generation $n$ has a Manhattan distance of no more than $n$ from the patriarch.
\end{prop}

\begin{proof}
The lineage of the cell is a path of length $n$ originating at the patriarch. Thus, the length of the shortest path does not exceed $n$.
\end{proof}

If we go back to Figure~\ref{fig:UWchildren}, we can describe the cells that have 3 children in the Square-UW automaton as cells which are at an even Manhattan distance from the patriarch.

\section{Pioneers and the Sierpinski triangle}

Define a \textit{pioneer} to be a cell born in generation $n$ with a Manhattan distance of $n$ from the patriarch. This means, the lineage of a pioneer represents the shortest path in the Manhattan sense. Pioneers move outward as quickly as possible and never turn back. The parent of a pioneer is also a pioneer. Moreover, all ancestors of a pioneer are pioneers. Pioneers are not afraid to go away and explore new territory.

The set of all pioneers has a relatively simple structure, as we will prove.

\begin{theorem}\label{thm:sierpinski}
The set of all pioneers in one slice of either automaton is exactly equivalent to the Sierpinski gasket.
\end{theorem}

\begin{proof}
We proceed by induction on the generation. For generation 0, the patriarch is a pioneer, and it matches the top cell of the Sierpinski gasket. Assume that all pioneers of generation $n$ correspond to the small triangles in the $(n+1)$-th row of the gasket, that is, they belong to generation $n$ of the gasket.

We now consider all pioneers in the automata that belong to generation $n+1$. The Sierpinski gasket can be characterized as Pascal's triangle modulo 2, so it satisfies a modulo-2 version of Pascal's identity. A cell is part of the Sierpinski gasket if and only if exactly one of the  neighbor cells above it are part of the gasket. By the inductive hypothesis these neighbors are part of the gasket if and only if they are pioneers.

Each cell in one of the automata has 1 or two neighbors that are closer to the patriarch. That means after generation $n$ is born, the cells on the grid that are at Manhattan distance $n+1$ from the patriarch cannot have more than two live neighbors. They are born when exactly one of the neighbors is alive; the automaton corresponding with the Sierpinski gasket follows the same rule.
\end{proof}

We have not only showed that the Sierpinski triangle is a part of both Ulam-Warburton automata, we have also showed that the gasket plays a special role. Its cells are pioneers. We now want to show this on a formulaic level. The description of cells that eventually belong to the Square-UW automaton are known \cite{APS, PSX}. We need a definition to describe it. 

Given an integer $n$, the \textit{2-adic order} of $n$ is the highest power of 2 that divides $n$. For example, the 2-adic order of an odd number is 0. A cell in the square grid will eventually belong to the Square-UW automaton if the coordinates of its center have different 2-adic orders. The cell with the center $(x,y)$ corresponds to the $\binom{x+y}{x}$ in the Pascal's triangle. The corresponding cell belongs to the Sierpinski triangle, when the corresponding binomial coefficient is odd. By Kummer's theorem, the binomial coefficients are odd if and only if summing $x$ and $y$ in binary does not produce a carry. In other words, $x$ and $y$ do not both have a 1 in the same binary unit. It follows that their 2-adic orders are different. Thus the gasket belongs to the square automaton.

\section{Acknowledgements}

We are grateful to Richard Stanley for suggesting this project. We are also grateful to MIT-PRIMES program for sponsoring this research.

\end{document}